\documentclass[letterpaper,11pt]{article}
\usepackage{amsmath,amssymb, amsthm}
\usepackage{textcomp,enumerate}
\usepackage{tikz,float,graphicx}
\usepackage{xcolor}
\usepackage{comment}
\usepackage{fullpage}
\usepackage{float}
\usepackage{subcaption}
\usepackage{hyperref}
\usepackage{thmtools}

\theoremstyle{plain}
\newtheorem{theorem}{Theorem}
\newtheorem{proposition}[theorem]{Proposition}
\newtheorem{corollary}[theorem]{Corollary}
\newtheorem{lemma}[theorem]{Lemma}
\newtheorem{conjecture}[theorem]{Conjecture}
\theoremstyle{definition}
\newtheorem{definition}[theorem]{Definition}
\newtheorem{question}[theorem]{Question}

\theoremstyle{remark}
\newtheorem*{observation}{Observation}
\newtheorem{remark}[theorem]{Remark}

\declaretheoremstyle[%
  spaceabove=0pt,%
  spacebelow=0pt,%
  headfont=\normalfont\itshape,%
  postheadspace=1em,%
  qed=\qedsymbol%
]{mystyle} 
\declaretheorem[name={Proof},style=mystyle,unnumbered,
]{prf}

\renewcommand{\le}{\ensuremath \leqslant}
\renewcommand{\ge}{\ensuremath \geqslant}
\DeclareMathOperator{\ex}{ex}

\newcommand{\YP}[1]{{\color{magenta!80!black} \small{\textbf{Yani:~}#1}}}

\newcommand{\col}{black}
\newcommand{\arc}[2]{\draw[draw=\col,thick] (#1) edge[in=90,out=90] (#2);}
\newcommand{\arcd}[2]{\draw[draw=\col,thick] (#1) edge[in=270,out=270] (#2);}

\newcommand{\isovertex}[1]{\draw[fill=black] (#1) ellipse (0.15cm and 0.1cm);}

\newcommand{\OrderingBoilerplate}[1]{
\begin{tikzpicture}[yscale = 1.5,scale=0.25]\foreach \n in {1,...,10} {\coordinate (\n) at (\n,0);}#1\end{tikzpicture}
}

\tikzset{color1/.style={color=red,thick,opacity=1}}
\tikzset{color2/.style={color=cyan,thick,opacity=1}}

\newcommand{\cA}{\ensuremath{\mathcal A}}
\newcommand{\cB}{\ensuremath{\mathcal B}}
\newcommand{\cC}{\ensuremath{\mathcal C}}
\newcommand{\cP}{\ensuremath{\mathcal P}}

\newcommand{\cH}{\ensuremath{\mathcal H}}
\newcommand{\cG}{\ensuremath{\mathcal G}}
\newcommand{\cK}{\ensuremath{\mathcal K}}

\newcommand{\cM}{\ensuremath{\mathcal M}}
\newcommand{\cR}{\ensuremath{\mathcal R}}
\newcommand{\PP}{\ensuremath{\mathbb P}}

\title{Ramsey numbers of ordered graphs under graph operations}
\author{Jesse Geneson\thanks{Department of Mathematics, Iowa State University,
\texttt{geneson@iastate.edu}, \texttt{icwass@iastate.edu}} \and 
Amber Holmes \thanks{Department of Mathematics, University of Kentucky, \texttt{amber.holmes@uky.edu}} \and 
Xujun Liu \thanks{Department of Mathematics, University of Illinois at Urbana-Champaign, \texttt{xliu150@illinois.edu}, \texttt{dn2@illinois.edu}.  Research of Xujun Liu is supported by Award RB17164 of the Research Board of UIUC.} \and 
Dana Neidinger\footnotemark[3] \and 
Yanitsa Pehova\thanks{Mathematics Institute, University of Warwick, \texttt{y.pehova@warwick.ac.uk}. This  work  has  received  funding  from  the  European  Research  Council  (ERC)  under  the  European  Union’s Horizon 2020 research and innovation programme (grant agreement No 648509).  This publication reflects only its authors’ view; the European Research Council Executive Agency is not responsible for any use that may be made of the information it contains. } \and 
Isaac Wass\footnotemark[1]
}

\begin{document}
\maketitle

\begin{abstract}
An ordered graph $\cG$ is a simple graph together with a total ordering on its vertices. The (2-color) Ramsey number of $\cG$ is the smallest integer $N$ such that every 2-coloring of the edges of the complete ordered graph on $N$ vertices has a monochromatic copy of $\cG$ that respects the ordering.

In this paper we investigate the effect of various graph operations on the Ramsey number of a given ordered graph, and detail a general framework for applying results on extremal functions of 0-1 matrices to ordered Ramsey problems. We apply this method to give upper bounds on the Ramsey number of ordered matchings arising from sum-decomposable permutations, an alternating ordering of the cycle, and an alternating ordering of the tight hyperpath. 
We also construct ordered matchings on $n$ vertices whose Ramsey number is $n^{q+o(1)}$ for any given exponent $q\in(1,2)$.
\end{abstract}

\section{Introduction}

An \emph{ordered graph} $\cG$ is a pair $(G,<)$ where $G$ is a simple graph and $<$ is a linear ordering on the vertices. In this paper we look at a natural extension of the Ramsey number for simple graphs to the ordered setting: the \emph{Ramsey number $R\left(\cG\right)$ of an ordered graph} $\cG$ is the minimum number of vertices in an ordered complete graph $\cK_n$ such that every coloring of its edges contains a monochromatic copy of $\cG$. 
(An ordered graph $\cH$ is \emph{contained} in an ordered graph $\cG$ if there is an order-preserving injection from the vertices of $\cH$ to the vertices of $\cG$ that preserves edges.) 

Note that by Ramsey's theorem the Ramsey number of an ordered graph is well-defined. Indeed, if we let $R=R(K_{|\cG|})$, where $K_n$ denotes the complete unordered graph, then every coloring of $\cK_R$ contains a monochromatic copy of $\cK_{|\cG|}$, and hence a copy of $\cG$.

Ordered Ramsey theory has become increasingly popular in recent years. In a series of papers Balko et al. \cite{balko2015,balko2016+}, and independently Conlon et al. \cite{Conlon2017} investigated connections between some natural ordered graph parameters and the corresponding Ramsey numbers, as well as some striking differences between Ramsey numbers of classical and ordered sparse graphs.

For example, if we bound the \emph{bandwidth} of an ordered graph $\cG$, that is, the length of the longest edge in the ordering of $V(\cG)$, then its Ramsey number will be a power of $v(\cG)$, proportional to its bandwidth:

\begin{theorem}[Balko, Cibulka, Kr\'al, Kyn\v{c}l \cite{balko2015}]
For every fixed positive integer $k$, there is a constant $C_k$ such that every $n$-vertex ordered graph $\cG$ with bandwidth $k$ satisfies
\[R(\cG)\le C_k n^{128k}.\]
\end{theorem}

Another natural graph parameter that has been considered in \cite{balko2015} and \cite{Conlon2017} is the \emph{interval chromatic number} $\chi_I(\cG)$, which is the smallest integer $k$ such that the vertices of an ordered graph $\cG$ can be properly colored in $k$ colors, where each color class consists of consecutive integers in the ordering of $V(\cG)$. 
If $\chi_I(\cG) = k$, we say $\cG$ is \emph{interval $k$-chromatic}.
Then, the following holds:

\begin{theorem}[Conlon, Fox, Lee, Sudakov \cite{Conlon2017}]
There exists a constant $c$ such that for any ordered graph $\cG$ on $n$ vertices with degeneracy\footnote{The \emph{degeneracy} of an ordered graph $\cG$ is the degeneracy of the corresponding unordered graph; that is, the smallest integer $d$ such that there exists an ordering of its vertices in which each vertex $v$ has at most $d$ neighbours $w$ with $w<v$ in the ordering.} $d$ and interval chromatic number $\chi$, we have
\[R(\cG)\le n^{cd \log \chi}.\]
\end{theorem}

However, the link between the interval chromatic number of a given ordered graph $\cG$ and its Ramsey number is not clear. Indeed, there exist interval 2-chromatic orderings $\cM_n$ of the matching on $n$ vertices, whose Ramsey number is of the order $n^{2-o(1)}$ \cite{Conlon2017,balko2016+}, while it is well-known that the Ramsey number $R(M_n)$ of a matching on $n$ vertices is linear in $n$. In fact, some ordered matchings (independent of their interval-chromatic number) have much higher Ramsey number:

\begin{theorem}[Conlon, Fox, Lee, Sudakov \cite{Conlon2017}]\label{thm:randommatching}
There exists a positive constant $c$ such that for all even $n$ there exists an ordered matching $\cM$ on $n$ vertices such that
\[R(\cM)\ge n^{c\log n/\log \log n}\]
\end{theorem}

Despite that, there are some orderings of $M_n$ which have linear Ramsey number. For example under the canonical ordering $0<1<2<...<n$ (for odd $n$), the ordered matching whose set of edges is $\{i,n-i\}_{i=1}^{(n-1)/2}$ has Ramsey number at most $2n+1$ by a simple pigeonhole argument. This prompts the question of minimizing the Ramsey number over all orderings of a given simple graph $G$. There has been some recent progress on this question by Balko, Jel\'inek and Valtr \cite{balko2016+} who showed that every graph $G$ on $n$ vertices with maximum degree $2$ admits an ordering whose Ramsey number is linear in $n$.

In this paper we investigate the behavior of Ramsey numbers of ordered graphs under certain graph operations such as taking disjoint unions and adding single edges and vertices. We also employ the method of matrix extremal functions to give bounds on the Ramsey numbers of specific orderings of the cycle and matchings, and we also give extensions to hypergraphs.

\textbf{Notation.} Throughout this paper, letters in cursive correspond to ordered graphs. We denote by $\overleftarrow{\cG}$ the \emph{mirror} of the ordered graph $\cG$, which is obtained by reversing the ordering of $V(\cG)$. The vertex set of a graph $\cG$ is denoted by $V(\cG)$, the number of vertices being $v(\cG)$, and its edge set is denoted by $E(\cG)$, the number of its edges being $e(\cG)$.

\section{Graph operations and Ramsey numbers}

In this section we present some results on the effects of various graph operations on the Ramsey number of an ordered graph.

\subsection{Disjoint union}

Let $\cG$ and $\cH$ be ordered graphs. We denote by $\cG+\cH$ the ordered graph on $v(G) + v(H)$ vertices where the first $v(\cG)$ vertices form a copy of $\cG$ and the remaining $v(\cH)$ vertices form a copy of $\cH$.
Note that $\cG+\cH$ is a specific ordering of the disjoint union $G'+H'$ of the unordered graph $G'$ underlying $\cG$ and the unordered graph $H'$ underlying $\cH$. Also note that $\cG+\cH$ is not always isomorphic or mirror-symmetric to $\cH+\cG$ and thus they may have different Ramsey numbers. 

Indeed, for example $R(\OrderingBoilerplate{\arc{2}{3};\arc{4}{5};\isovertex{1};\isovertex{6};}) = 8$ and $R(\OrderingBoilerplate{\arc{1}{2};\arc{5}{6};\isovertex{3};\isovertex{4};})=10$. The upper bounds for these two graphs follow from a routine pigeonhole principle argument, and the lower bounds follow from the colorings in Figures \ref{fig:1} and \ref{fig:2}.

\begin{minipage}{0.5\textwidth}
	\begin{figure}[H]
		\centering
		\begin{tikzpicture}[xscale=0.9, yscale = 1.5]
		\foreach \x in {1,...,7}{
			\node[circle] (\x) at (\x,0) {\x}; 
		}
        \renewcommand{\col}{red}
        \foreach \x/\y in {2/3,2/4,3/4}{
        	\arc{\x}{\y}
        }
        \renewcommand{\col}{cyan}
        \foreach \x/\y in {4/5,4/6,5/6}{
        	\arcd{\x}{\y}
        }
		\end{tikzpicture}
		\caption{An edge 2-coloring of $K_7$ that avoids a monochromatic copy of $\protect\OrderingBoilerplate{\protect\arc{2}{3};\protect\arc{4}{5};\protect\isovertex{1};\protect\isovertex{6};}$. Uncolored edges can be any color.}
		\label{fig:1}
	\end{figure}
\end{minipage}\quad
\begin{minipage}{0.45\textwidth}
	\begin{figure}[H]
		\centering
		\begin{tikzpicture}[xscale=0.9, yscale = 1.5]
		\foreach \x in {1,...,9}{
			\node[circle] (\x) at (\x,0) {\x}; 
		}
        \renewcommand{\col}{red}
        \foreach \x/\y in {1/2,1/3,1/4,1/5,2/3,2/4,2/5,3/4,3/5,4/5}{
        	\arc{\x}{\y}
        }
        \renewcommand{\col}{cyan}
        \foreach \x/\y in {5/6,5/7,5/8,5/9,6/7,6/8,6/9,7/8,7/9,8/9}{
        	\arcd{\x}{\y}
        }
		\end{tikzpicture}
		\caption{An edge 2-coloring of $K_9$ that avoids a monochromatic copy of $\protect\OrderingBoilerplate{\protect\arc{1}{2};\protect\arc{5}{6};\protect\isovertex{3};\protect\isovertex{4};}$. Uncolored edges can be any color.}
		\label{fig:2}
	\end{figure}
\end{minipage}

\begin{lemma}\label{ordered-disjoint-union}
For ordered graphs $\cG$ and $\cH$, the following inequalities hold:
$$R(\cG) + R(\cH) \le R(\cG + \cH) \le R(\cG)+R(\cH)+R(\cH + \cG).$$
\end{lemma}

\begin{prf}
Let us denote $k=R(\cG)-1$ and $l=R(\cH)-1$, and suppose we are given the complete graph on vertex set $\{u_1,u_2,\dots,u_{k},v,w_1,w_2,\dots,w_{l}\}$. By definition, there is a coloring of $K_{k}$ that is $\cG$-free and a coloring of $K_{l}$ that is $\cH$-free.

Color the edges among $\{u_1,\dots,u_{k}\}$ using this $\cG$-free coloring, and color the edges among $\{w_1,\dots,w_{l}\}$ with this $\cH$-free coloring. By construction we see that the rightmost vertex of any copy of $\cG$ must belong to $\{v,w_1,w_2,\dots,w_{l}\}$ and the leftmost vertex of any copy of $\cH$ must belong to $\{u_1,u_2,\dots,u_{k},v\}$.

Suppose this graph has a monochromatic copy of $\cG + \cH$. Then the rightmost vertex of the copy of $\cG$ is left of the leftmost vertex of the copy of $\cH$. However, this implies that a vertex in $\{v,w_1,w_2,\dots,w_{l}\}$ is left of a vertex in $\{u_1,u_2,\dots,u_{k},v\}$, a contradiction. This establishes the lower bound.

Now consider any 2-edge-coloring of $K_{k+l+m+2}$ where $m=R(\cH+\cG)$. The first $k+1$ vertices contain a monochromatic $\cG$, the next $m$ vertices contain a monochromatic $\cH + \cG$, and the remaining $l+1$ vertices contain a monochromatic $\cH$. Among these three structures, two are the same color, giving the desired monochromatic copy of $\cG + \cH$.
\end{prf}

\begin{proposition}
Let $\cG$ be an ordered graph, and let $\cH$ be a subgraph of $\cG$. Then $R(\cG+\cH) \le 2R(\cG)+R(\cH)$ and $R(\cH+\cG) \le 2R(\cG)+R(\cH)$ .
\end{proposition}


\begin{prf}
Consider a 2-edge coloring of the ordered complete graph on $2R(\cG)+R(\cH)$ vertices.
 
By definition, the first $R(\cG)$ vertices and the middle $R(\cG)$ contain a monochromatic copy of $\cG$ and the last $R(\cH)$ vertices contain a monochromatic copy of $\cH$. Out of these two monochromatic copies of $\cG$ and one monochromatic copy of $\cH$, two are the same color. Since $\cH$ is contained in $\cG$, these two copies necessarily form a monochromatic copy of $\cG+\cH$. 

Analogously, if we instead wish to find a monochromatic copy of $\cG+\cH$, consider the first $R(\cH)$ vertices, the middle $R(\cG)$ and the last $R(\cG)$ vertices. 
\end{prf} 

By an analogous argument, for any pair of graphs $\cG$ and $\cH$ we can obtain the following upper bound:

\begin{proposition}
Let $\cG$ and $\cH$ be any two ordered graphs. Then $R(\cG+\cH) \le R(\cG)+R(\cH)+R(K_n)$, where $n = \max(R(\cG),R(\cH))$.
\end{proposition}

\begin{prf}
The first $R(\cG)$ vertices contain a monochromatic copy of $\cG$, the middle $R(K_n)$ contain a monochromatic copy of $\cG$ and a monochromatic copy of $\cH$, both in the same color, and the last $R(\cH)$ vertices in the ordering contain a monochromatic copy of $\cH$. Thus there must be a monochromatic copy of $\cG+\cH$. 
\end{prf}

\subsection{Adding isolated vertices}

How much larger can the Ramsey number get if we add a single isolated vertex to an ordered graph?

\begin{proposition} \label{vertex on end}
Let $\cG$ be any ordered graph. Let $\cG'$ be the ordered graph on $v(\cG)+1$ vertices obtained from adding a single isolated vertex either to the right or to the left of all vertices of $\cG$. Then $R(\cG') = R(\cG)+1$. 
\end{proposition}

\begin{prf} Without loss of generality, we may assume that we are adding an isolated vertex to the right of all vertices of $\cG$.

First, let us consider a 2-edge colored ordered complete graph on $R(\cG)+1$ vertices. By definition, there exists a monochromatic copy of $\cG$ in the first $R(\cG)$ vertices, and thus the last vertex along with this copy of $\cG$ form a monochromatic copy of $\cG'$.

Now consider a coloring of the complete graph on $R(\cG)-1$ vertices which avoids a monochromatic copy of $\cG$. Add a vertex to the right of this, and color all of its adjacent edges arbitrarily. Any monochromatic copy of $\cG'$ must contain a copy of $\cG$ on the first $R(\cG)-1$ vertices. Hence we have that $R(\cG') \ge R(\cG)+1$.
\end{prf}

We now consider a more general variant of this, where isolated vertices may be inserted anywhere in the vertex ordering of a graph.

\begin{definition}
For an ordered graph $\cG$ with vertex set $(v_1,...,v_n)$, we define the $(k_1,...,k_{n-1})$\emph{-spread} of $\cG$ to be the ordered graph $\cG'$ obtained by adding $k_i$ isolated vertices between $v_i$ and $v_{i+1}$ for each $1\le i\le n-1$. The \emph{head} of $\cG'$ is the largest positive integer $h$ such that $k_1=...=k_{h-1}=0$ and the \emph{tail} of $\cG'$ is the largest positive integer $t$ such that $k_{n-t+1}=...=k_{n-1}=0$.
\end{definition}
\begin{proposition} \label{addvertexub}
Let $\cG'$ be the $(k_1,...,k_{n-1})$-spread of an ordered graph $\cG$ on $n$ vertices, such that $k_i\le k$ for all $1\le i\le n-1$. Then $R(\cG') \le R(\cG) + k(R(\cG)-h-t+1)$, where $h$ and $t$ are respectively the head and the tail of $\cG'$. 
\end{proposition}

\begin{prf}
Consider a 2-edge colored ordered complete graph on $R(\cG) + k(R(\cG)-h-t+1)$ vertices. Consider the set $X$ consisting of the first $h$ vertices, the last $t$ vertices, and for every vertex in between, starting with vertex $h$, every $(k+1)^\text{th}$ vertex (i.e. $h, h+k+1, \ldots$). Since $|X|=R(\cG)$, there is a monochromatic copy of $\cG$ on these vertices.

However, by our choice of "middle" vertices in $X$, this copy of $\cG$ can be extended to a  $(k,k,...,k)$-spread of $\cG$, thus giving a monochromatic copy of $\cG'$.
\end{prf}

\begin{observation}Note that this bound is in a certain sense tight for all ordered graphs $\cG$. Consider the ordered graph $\cG'$ obtained by inserting $k$ isolated vertices between every pair of consecutive vertices in the ordering of $V(\cG)$. Then $R(\cG')\ge (k-1)(R(\cG)-1)+1$. Indeed, let $c$ be a $\cG$-free coloring of the complete graph on $m=R(\cG)-1$ vertices. Consider intervals $I_1,...,I_m$, each consisting of $k-1$ vertices, and for each $1\le i<j\le m$ color all edges between $I_i$ and $I_j$ in color $c(ij)$. Then any monochromatic copy of $\cG$ in this complete graph must have two vertices in the same interval, which means it is not contained in a copy of $\cG'$. This, together with Proposition \ref{addvertexub} shows that $R(\cG')\sim kR(\cG)$, while the classical Ramsey number of a graph would grow only by an additive factor of $k(v(\cG)-1)$.
\end{observation}

\newpage

\subsection{Adding an edge}


\begin{proposition} \label{addlastedge}
Let $\cG$ be an $n$-vertex ordered graph with $n \ge 2$. Let $\cG'$ be obtained from $\cG$ by adding a single vertex after the last vertex of $\cG$, and adding a single edge from the rightmost vertex of $\cG$ to this new vertex. Then $R(\cG)+n \le R(\cG') \le R(\cG) + 2n - 1$.
\end{proposition}

\begin{figure}[H]
		\centering
		\begin{tikzpicture}[yscale = 1]
		\coordinate (1) at (1,0);
		\coordinate (2) at (2,0);
		\coordinate (3) at (3,0);
		\coordinate (4) at (4,0);
        \coordinate (5) at (5,0);
		
		\draw (2.5,0) circle [x radius=2, y radius=1];
        
        \node (G) at (2.5,.5) {$G$};
        
        \arc{4}{5}
        
        \foreach \n in {1,...,5} {
			\draw[fill=black] (\n) circle (0.05cm);
		}
		\end{tikzpicture}
		\caption{Graph $\cG'$ created from $\cG$ by adding one vertex and one edge.}
		\label{separate}
	\end{figure}
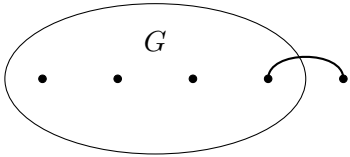
    
\begin{prf}
Consider an arbitrary 2-edge-coloring of the complete ordered graph $\cK_N$ on $N=R(\cG) + 2n - 1$ vertices. Let $V$ denote its vertex set. By definition the first $R(\cG)$ vertices of $V$ contain a monochromatic copy of $\cG$. Let $v_1$ denote the rightmost vertex of this copy of $\cG$. Then the first $R(\cG)$ vertices of $V \backslash \{v_1\}$ must contain a monochromatic copy of $\cG$ distinct from the first copy of $\cG$ (since the first copy of $\cG$ was incident to $v_1$, and this new copy cannot be incident to $v_1$). Let $v_2$ be the rightmost vertex of this new copy of $\cG$. Again, the first $R(\cG)$ vertices of $V \backslash \{v_1,v_2\}$ contain a distinct copy of $\cG$. We can continue in this fashion to obtain $2n-1$ distinct monochromatic copies of $\cG$ and vertices $\{v_1,\dots,v_{2n-1}\}$ that are the rightmost vertex of the copies of $\cG$, such that the last vertex of $\cK_N$ is disjoint from each copy of $\cG$.

By pigeonhole principle, $n$ of these graphs are the same color. Without loss of generality they are red and their rightmost vertices are $v_1,...,v_n$, so that $v_i$ is to the left of $v_{i+1}$ for each $i \in [n-1]$. Let $v_{n+1}$ denote the last vertex in $\cK_N$. Note that if some edge $v_iv_j$ is red, then the $i$-th copy of $G$ and $v_iv_j$ gives us a red copy of $\cG'$. Otherwise, if $v_iv_j$ is blue for all $1 \le i < j \le n+1$, then $\{v_1,\dots,v_n,v_{n+1}\}$ form a blue copy of $\cK_{n+1}$, which contains a copy of $\cG'$. \end{prf}

We remark that this bound is tight for the monotone path\footnote{The \emph{monotone path} $\cP_n^{mon}$ has vertex set $1<2<...<n$ and edge set $\{j(j+1):1\le j \le n-1\}$.}. Indeed, it is known that  $R(\cP_n^\text{mon}) = (n-1)^2+1$ (see \cite{fox2012}), and adding a single edge at the end gives us $\cP_{n+1}^\text{mon}$. The difference between their Ramsey numbers is exactly $n^2+1-(n-1)^2-1=2n-1.$

\begin{proposition} \label{addlastedge-lower}
Let $\cG$ and $\cG'$ be as described in the previous proposition, with the additional constraints that $n \ge 2$ and that the rightmost vertex of $\cG$ is not an isolated vertex.

Then $R(\cG)+n \le R(\cG')$.
\end{proposition}

\begin{proof}
We proceed by construction. Consider the complete ordered graph $\cK_N$ on $N = R(\cG)+n-1$ vertices. Let the first $R(\cG)-1$ vertices belong to set $A$ and let the last $n$ vertices belong to set $B$. Let $E_A$ be the set of edges contained in $A$, let $E_B$ be the set of edges contained in $B$, and let $E_C$ be the edges between $A$ and $B$. By definition we can color the edges in $E_A$ so that no monochromatic copy of $\cG$ appears among $A$. Color all edges in $E_B$ red and color all edges in $E_C$ blue.

Assume for sake of contradiction that we have a monochromatic copy of $\cG'$, denoted $H$. Let $w$ denote the rightmost vertex in $H$, let $v$ denote the second-rightmost vertex in $H$, and let $u$ be the vertex left of $v$ that is incident to $v$. Clearly $H-\{w\}$ (a copy of $\cG$) cannot be contained in $A$ due to the counter-coloring, so $v \in B$. Since $w$ is right of $v$, then $w \in B$ as well. Additionally, $n \ge 2$ so $u \neq v$, and $H$ cannot be contained in $B$ since $|H| = n+1 > n = |B|$, so $u \in A$. We now have $uv$ is a blue arc and that $vw$ is a red arc, contradicting the assumption that $H$ was monochromatic.
\end{proof}

\section{Ramsey numbers of ordered graphs via extremal functions of matrices}

Balko et al \cite{balko2015} used bounds on extremal functions of forbidden 0-1 matrices to show that the Ramsey number of the alternating path of size $n$ grows linearly with $n$, and Neidinger and West \cite{Neidinger2018+} generalized these results to stitched ordered graphs. In this section, we generalize their method to obtain size-linear bounds on the Ramsey numbers of many other families of ordered graphs.

Extremal functions of forbidden 0-1 matrices have been applied to a variety of other problems, including bounding the complexity of an algorithm for finding a minimal rectilinear path in a grid with obstacles \cite{mitchell} in the first paper on the topic. They have also been used to bound the maximum number of unit distances in a convex $n$-gon \cite{ngon} and the maximum possible lengths of sequences avoiding forbidden subsequences \cite{Pettie2011}. The most well-known application of 0-1 matrix extremal functions is Marcus and Tardos' solution to the Stanley-Wilf conjecture \cite{marcus2004}, which used linear upper bounds on extremal functions of forbidden permutation matrices to prove exponential upper bounds on the number of permutations of $[n]$ that avoid a given forbidden permutation. 

\subsection{Method outline}

Let $\cG$ be an interval 2-chromatic ordered graph with vertex set $V=I_1\cup I_2$. Its \emph{associated matrix} $P_\cG$ is the $|I_1|$-by-$|I_2|$ matrix defined by
\[P_\cG(i,j)=\begin{cases} 1 &\mbox{ if }ij\in E(G)\\ 0 &\mbox{ otherwise}\end{cases}.\]

We say that a zero-one matrix $A$ \emph{contains} another zero-one matrix $P$ if some submatrix of $A$ is either equal to $P$ or can be turned to $P$ by changing some ones to zeroes. Otherwise, we say that $A$ \emph{avoids} $P$. We define the extremal function $\ex(n, P)$ to be the maximum number of ones in an $n$-by-$n$ zero-one matrix that avoids $P$. In other words, $\ex(n, P)$ is one less than the minimum number of ones that force an $n$-by-$n$ zero-one matrix to contain $P$. 

With these two definitions in mind, we can now give an upper bound for the Ramsey number of an ordered graph $\cG$ in terms of the extremal number of its associated matrix $P_\cG$.

\begin{lemma}\label{lem:01mat:ramsey}
For any ordered interval 2-chromatic graph $\cG$, it holds that $$\left(R(\cG\right)-1)^2\le 8 \ex\left((R(\cG)-1)/2, P_\cG\right).$$
\end{lemma}
\begin{prf}
Consider a $\cG$-free 2-coloring of the ordered complete graph $\cK_N$ on $N=R(\cG)-1$ vertices for some integer $N$. Split the vertex set into two intervals $I_1$ and $I_2$ of length $\lfloor N/2\rfloor$ and $\lceil N/2\rceil$, and let $\cH=G[I_1,I_2]$ be the interval 2-chromatic graph induced on $I_1\cup I_2$. Without loss of generality at least half of its edges are red. Then the $N/2$-by-$N/2$ matrix $P_\cH$ associated to the red graph has at least $N^2/8$ ones and is $\cG$-free. Hence $N^2/8\le \ex(N/2,P_\cG)$.
\end{prf}

\subsection{Matrix operations and extremal functions}

Below, we exhibit several operations that can be performed on a fixed 0-1 matrix $P$ with extremal function linear in $n$ to yield a new 0-1 matrix such that this linearity is preserved. Then the resulting interval 2-chromatic ordered graph has Ramsey number that is linear in the number of vertices. Most of these operations are from existing literature, but for our first operation, we need a tighter bound than was given in the literature, which we will prove below.

The proof below is similar to the original bound from Tardos \cite{tardos2005} which gave a bound of $\ex(n, P') \le k^2 \ex(n, P)+4 k n$, but we modify the method to get a better coefficient for $\ex(n, P)$. 

\begin{figure}[H]
    \centering
\caption{}
\label{fig:matrices}
    \begin{subfigure}[H]{0.3\textwidth}
        \centering
        \includegraphics{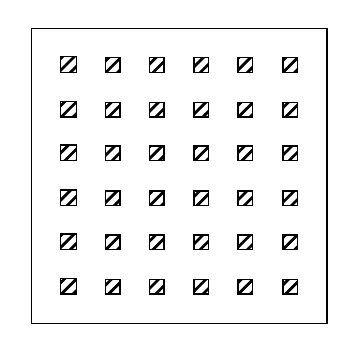}
        \caption{Operation for Lemma \ref{addblanks}}
        \label{fig:addblanks}
    \end{subfigure}
    ~ 
    \begin{subfigure}[H]{0.3\textwidth}
        \centering
        \includegraphics{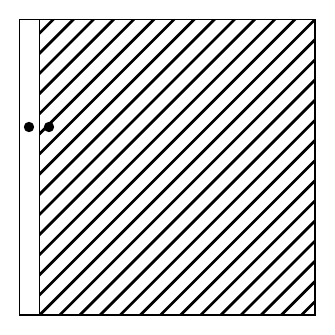}
        \caption{Operation for Lemma \ref{addlast}}
        \label{fig:addlast}
    \end{subfigure}
    \vspace{20pt}
    \begin{subfigure}[H]{0.3\textwidth}
        \centering
        \includegraphics{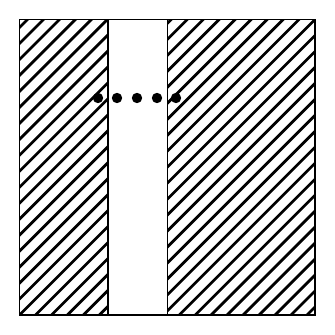}
        \caption{Operation for Lemma \ref{addmid}}
        \label{fig:addmid}
    \end{subfigure}
    ~
    \begin{subfigure}[H]{0.3\textwidth}
        \centering
        \includegraphics{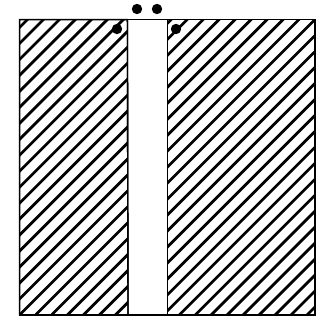}
        \caption{Operation for Lemma \ref{addmidup}}
        \label{fig:addmidup}
    \end{subfigure}
    \begin{subfigure}[H]{0.3\textwidth}
        \centering
        \includegraphics{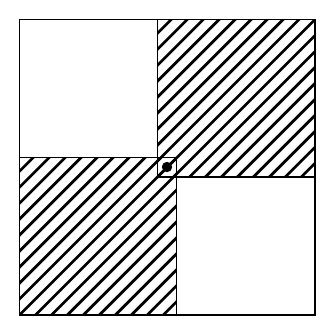}
        \caption{Operation for Lemma \ref{diagatt}}
        \label{fig:diagatt}
    \end{subfigure}

\end{figure}

\begin{lemma} \label{addblanks}
Suppose that the matrix $P'$ is obtained from $P$ by inserting $k-1$ empty rows between every adjacent pair of rows in $P$, and $k-1$ empty columns between every adjacent pair of columns in $P$, as well as $k-1$ empty rows and $k-1$ empty columns before the first rows and columns and after the last rows and columns (see Figure \ref{fig:addblanks}). Then $\ex(n, P) \le \ex(n, P') \le k \ex(n, P) + 6 k n$.
\end{lemma}

\begin{prf}
The first inequality follows from the fact that $\ex(n, M') \le \ex(n, M)$ whenever $M$ contains $M'$. For the second inequality, suppose that $A$ is a maximal weight $n \times n$ matrix that avoids $P'$. 
First delete any ones in the first and last $k$ rows and columns of $A$, and then delete the ones from up to an additional $k$ rows and columns so that the remaining ones in $A$ are contained in an $m \times n$ submatrix $B$ with $m$ divisible by $k$. At most $6 k n$ ones are deleted in this process.

For $0 \le a, b < k$, let $B_{a b}$ be the submatrix obtained from $B$ by restricting to the row and column indices $i$ and $j$ satisfying $i \mod k = a$ and $j \mod k = b$. The sum of the number of ones in all the submatrices $B_{a b}$ is the number of ones in $B$. Moreover every submatrix $B_{a b}$ must avoid $P$, or else $A$ would contain $P'$. 

Putting everything together, we have $\ex(n, P') = w(A) \le w(B) + 6 k n \le k^2 \ex(m/k, P)+6 k n$. Using the result of Pach and Tardos \cite{Pach2006} that $\ex(n, P)$ is super-additive, we have $\ex(n, P') \le k \ex(m, P)+6 k n \le k \ex(n, P) + 6 k n$.
\end{prf}

Now we mention some relevant results from the literature with citations, but without proof.

\begin{lemma}[\cite{tardos2005}]\label{addlast}
Suppose that the matrix $P'$ is obtained from $P$ by adding a first column to $P$ with a single one next to a one of $P$ (see Figure \ref{fig:addlast}). Then $\ex(n, P) \le \ex(n, P') \le \ex(n, P)+n$.
\end{lemma}

\begin{lemma}[\cite{tardos2005}] \label{addmid}
Suppose that the matrix $P'$ is obtained from $P$ by adding $t$ extra columns to $P$ between two columns of $P$, such that each column has a single one in the same row and the newly introduced ones have a one from $P$ next to them on both sides (see Figure \ref{fig:addmid}), then $\ex(n, P) \le \ex(n, P') \le (t+1)\ex(n, P)$.
\end{lemma}

\begin{lemma}[\cite{Keszegh2009}]\label{addmidup}
Let $P$ be a 0-1 matrix with two ones in the first row in columns $m$ and $m+1$. Define $P'$ to be the matrix obtained from $P$ by adding two empty columns between $m$ and $m+1$, and then adding a new first row with exactly two ones, both in the new columns (see Figure \ref{fig:addmidup}). Then $\ex(n, P) \le \ex(n, P') \le c \ex(n, P)$ for some absolute constant $c$.
\end{lemma}

\begin{lemma}[\cite{Keszegh2009}]\label{diagatt}
Suppose that $P$ is a 0-1 matrix with a one in the top right corner, and $Q$ is a 0-1 matrix with a one in the bottom left corner. Let $R$ be the matrix obtained from $P$ and $Q$ by placing them so that the top right corner of $P$ is in the same location as the bottom left corner of $Q$, and the entries outside of $P$ and $Q$ are all zeroes (see Figure \ref{fig:diagatt}). Then $max(\ex(n, P), \ex(n, Q)) \le \ex(n, R) \le \ex(n, P)+\ex(n, Q)$. 
\end{lemma}

The next theorem is a corollary of the preceding properties of 0-1 matrices and Lemma \ref{lem:01mat:ramsey}. This theorem produces many families of ordered graphs with Ramsey numbers that grow linearly with their size. We use the notation $\cG_{P}$ for the ordered graph with associated matrix $P$.

\begin{theorem} \label{thmoper}
Let $P$ be a fixed 0-1 matrix such that $\ex(n, P) = O(n)$. Suppose that we perform one of the operations below to yield a new family of 0-1 matrices $\left\{P_{j}\right\}$. Then the Ramsey numbers of the ordered graphs $\cG_{P_{j}}$ grow linearly with respect to their size.

\begin{enumerate}
\item \label{addrowcol} Assume that $P$ has a one in its first or last row or column. Let $P_j$ be obtained from $P$ by iterating the operation in Lemma \ref{addlast} a total of $j$ times, so that the sum of the number of rows and columns in $P_j$ is $j$ more than the sum of the number of rows and columns in $P$.

\item \label{addmidrowcol} Assume that $P$ has two adjacent ones in some row. Let $P_j$ be obtained from $P$ by applying the operation in Lemma \ref{addmid} with $t = j$, so that the number of columns in $P_j$ is $j$ more than the number of columns in $P$.

\item \label{graftrowcol} Assume that $P$ has two adjacent ones in the top row. Let $P'$ be obtained from $P$ by applying the operation in Lemma \ref{addmidup} to obtain a new 0-1 matrix with two adjacent ones in the top row, and then let $P_j$ be obtained from $P'$ by applying the operation in Lemma \ref{addmid} with $t = j$ to those adjacent ones in the first row, so that the number of columns in $P_j$ is $j+2$ more than the number of columns in $P$.

\item \label{diag2} Assume that $P$ has ones in the bottom right corner and the top left corner. Let $P_j$ be obtained from $P$ by applying the operation in Lemma \ref{diagatt} a total of $j$ times with $P = Q$.

\item \label{blankrowcol} For this last operation, there are no additional requirements on $P$. Let $P_j$ be obtained from $P$ by inserting $j$ empty rows between every adjacent pair of rows in $P$, and $j$ empty columns between every adjacent pair of columns in $P$.
\end{enumerate}
\end{theorem}

\begin{figure}[H]
\centering
$
\centering
    \begin{bmatrix}
        0       & 1         & 1         & 0 \\
        1       & 0         & 0         & 1 \\
        0       & 0         & 1         & 0 \\
    \end{bmatrix}$
\caption{The matrix $F$}
\label{fig:fulekfig}
\end{figure}

Theorem \ref{thmoper} can be applied to any 0-1 matrix $P$ such that $\ex(n, P) = O(n)$ to generate a family of ordered graphs with linear Ramsey numbers. Many such matrices $P$ are known, but it is a major open problem to determine all such $P$ \cite{marcus2004}.

To see a typical application of the last theorem, define $F$ as the matrix in Figure \ref{fig:fulekfig}. Fulek \cite{fulek2009} proved that $\ex(n, F) = O(n)$ using known results on bar visibility graphs.

\begin{figure}[H]
\centering
$
\centering
    \begin{bmatrix}
        0       & 1         & 1         & 0         & 0 \\
        1       & 0         & 0         & 1         & 1 \\
        0       & 0         & 1         & 0         & 0 \\
    \end{bmatrix}
    \begin{bmatrix}
        0       & 1         & 1         & 0         & 0 \\
        1       & 0         & 0         & 1         & 1 \\
        0       & 0         & 1         & 0         & 0 \\
        0       & 0         & 1         & 0         & 0 \\
    \end{bmatrix}
    \begin{bmatrix}
        0       & 0         & 1         & 0         & 0 \\
        0       & 1         & 1         & 0         & 0 \\
        1       & 0         & 0         & 1         & 1 \\
        0       & 0         & 1         & 0         & 0 \\
        0       & 0         & 1         & 0         & 0 \\
    \end{bmatrix}$
\caption{The first few matrices of a family $\left\{F_n\right\}$ obtained by applying Theorem \ref{thmoper}.\ref{addrowcol} to $F$}
\label{addrowcolfig}
\end{figure}

If we apply the operation in Theorem \ref{thmoper}.\ref{addrowcol} to $F$, we can obtain a family $\left\{F_{n} \right\}$ with $F_{1}$, $F_{2}$, and $F_{3}$ pictured in Figure \ref{addrowcolfig}. Note that $F$ can generate multiple families $\left\{F_n \right\}$ since there are multiple ones in $F$ that are either in the first row, last row, first column, or last column. Regardless of which family we generate using Theorem \ref{thmoper}.\ref{addrowcol}, we always get $R(F_{n}) = \Theta(n)$.

\begin{figure}[H]
\centering
$
\centering
    \begin{bmatrix}
        0       & 1         & 1         & 1         & 0 \\
        1       & 0         & 0         & 0         & 1 \\
        0       & 0         & 0         & 1         & 0 \\
    \end{bmatrix}
    \begin{bmatrix}
        0       & 1         & 1         & 1         & 1         & 0 \\
        1       & 0         & 0         & 0         & 0         & 1 \\
        0       & 0         & 0         & 0         & 1         & 0 \\
    \end{bmatrix}
    \begin{bmatrix}
        0       & 1         & 1         & 1         & 1         & 1         & 0 \\
        1       & 0         & 0         & 0         & 0         & 0         & 1 \\
        0       & 0         & 0         & 0         & 0         & 1         & 0 \\
    \end{bmatrix}$
\caption{The first few matrices of the family $\left\{F_n\right\}$ obtained by applying Theorem \ref{thmoper}.\ref{addmidrowcol} to $F$}
\label{addmidrowcolfig}
\end{figure}

Alternatively, if we apply the operation in Theorem \ref{thmoper}.\ref{addmidrowcol} to $F$, we can obtain a family $\left\{F_{n} \right\}$ with $F_{1}$, $F_{2}$, and $F_{3}$ pictured in Figure \ref{addmidrowcolfig}. For this family, we also get $R(F_{n}) = \Theta(n)$. If we instead apply the operation in Theorem \ref{thmoper}.\ref{graftrowcol}, we get a similar looking family (Figure \ref{addgraftrowcolfig}) which also has $R(F_{n}) = \Theta(n)$.

\begin{figure}[H]
\centering
$
\centering
    \begin{bmatrix}
        0       & 0         & 1         & 1         & 1         & 0         & 0 \\         
        0       & 1         & 0         & 0         & 0         & 1         & 0 \\
        1       & 0         & 0         & 0         & 0         & 0         & 1 \\
        0       & 0         & 0         & 0         & 0         & 1         & 0 \\
    \end{bmatrix}
    \begin{bmatrix}
        0       & 0         & 1         & 1         & 1         & 1         & 0         & 0 \\  
        0       & 1         & 0         & 0         & 0         & 0         & 1         & 0 \\
        1       & 0         & 0         & 0         & 0         & 0         & 0         & 1 \\
        0       & 0         & 0         & 0         & 0         & 0         & 1         & 0 \\
    \end{bmatrix}
    \begin{bmatrix}
        0       & 0         & 1         & 1         & 1         & 1         & 1         & 0         & 0 \\ 
        0       & 1         & 0         & 0         & 0         & 0         & 0         & 1         & 0 \\
        1       & 0         & 0         & 0         & 0         & 0         & 0         & 0         & 1 \\
        0       & 0         & 0         & 0         & 0         & 0         & 0         & 1         & 0 \\
    \end{bmatrix}$
\caption{The first few matrices of the family $\left\{F_n\right\}$ obtained by applying Theorem \ref{thmoper}.\ref{graftrowcol} to $F$}
\label{addgraftrowcolfig}
\end{figure}

We save an application of Theorem \ref{thmoper}.\ref{diag2} for the subsection on layered permutations, but the first two elements of the family obtained by applying Theorem \ref{thmoper}.\ref{blankrowcol} are pictured in Figure \ref{blankrowcolfig}.

\begin{figure}[H]
\centering
$
\centering
    \begin{bmatrix}
        0       & 0       & 1       & 0         & 1     & 0         & 0 \\
        0       & 0       & 0       & 0         & 0     & 0         & 0 \\
        1       & 0       & 0       & 0         & 0     & 0         & 1 \\
        0       & 0       & 0       & 0         & 0     & 0         & 0 \\
        0       & 0       & 0       & 0         & 1     & 0         & 0 \\
    \end{bmatrix}
    \begin{bmatrix}
        0       & 0     & 0       & 1       & 0     & 0         & 1     & 0     & 0         & 0 \\
        0       & 0     & 0       & 0       & 0     & 0         & 0     & 0     & 0         & 0 \\
        0       & 0     & 0       & 0       & 0     & 0         & 0     & 0     & 0         & 0 \\
        1       & 0     & 0       & 0       & 0     & 0         & 0     & 0     & 0         & 1 \\
        0       & 0     & 0       & 0       & 0     & 0         & 0     & 0     & 0         & 0 \\
        0       & 0     & 0       & 0       & 0     & 0         & 0     & 0     & 0         & 0 \\
        0       & 0     & 0       & 0       & 0     & 0         & 1     & 0     & 0         & 0 \\
    \end{bmatrix}$
\caption{The first two matrices of the family $\left\{F_n\right\}$ obtained by applying Theorem \ref{thmoper}.\ref{blankrowcol} to $F$}
\label{blankrowcolfig}
\end{figure}

\subsection{A note on ordered matchings}

An unordered matching on $n$ vertices has Ramsey number linear in $n$ but, for example, the authors of \cite{Conlon2017} showed that this can grow to $n^{\Omega(\log n/\log \log n)}$ in the ordered equivalent. Even for an interval 2-chromatic matching $\cM$ we have $R(\cM)\ge n^{2-o(1)}$. In this section we give a special class of matchings whose Ramsey numbers are linear, and, for each $q\in (1,2)$, we construct an interval 2-chromatic matching on $n$ vertices whose Ramsey number is equal to $n^{q+o(1)}$.

\begin{definition}
The \emph{direct sum} $A\oplus B$ of two matrices $A$ and $B$ is the block matrix \[\left[
\begin{array}{c|c}
0 & B \\
\hline
A & 0
\end{array}
\right]\]
We call a permutation matrix $P$ \emph{sum-decomposable} if it is a direct sum of permutation matrices called the \emph{blocks} of $P$. 
\end{definition}

We use some results about extremal numbers of permutation matrices to give an upper bound for the Ramsey number of interval 2-chromatic matchings arising\footnote{Note that a matrix $P$ is a permutation matrix if and only if the ordered graph $\cG_P$ is an interval 2-chromatic matching.} from sum-decomposable permutations. There has been some extensive work on deriving bounds for the extremal number of permutation matrices. Most notably, Marcus and Tardos \cite{marcus2004} proved that $\ex(n, P) = O(n)$ for every permutation matrix $P$, and Fox \cite{Fox2013} later sharpened the bound by proving that $\ex(n, P) = 2^{O(k)} n$ for every $k$-by-$k$ permutation matrix $P$. 

\begin{proposition}\label{prop:layeredperm}
Let $P$ be an $km$-by-$km$ sum-decomposable permutation matrix where each block is $k$-by-$k$. Then the Ramsey number of $\cG_{P}$ is $2^{O(k)}m$.
\end{proposition}

\begin{prf}
Let $P=P_1\oplus \cdots \oplus P_m$ and for $1\le i \le m$ let $Q_i=I_1\oplus P_i\oplus I_1$ be the $(k+2) \times (k+2)$ permutation matrix obtained by concatenating a single one, the permutation matrix $P_i$, and another single one. Let $Q$ be the matrix obtained by gluing $Q_1, \ldots ,Q_m$ as in Lemma \ref{diagatt}. Then by Lemma \ref{diagatt} $\ex(n,Q)\le \sum_i \ex(n,Q_i)\le m 2^{O(k)}n$ and by Lemma \ref{lem:01mat:ramsey} $R(\cG_{Q})\le 2^{O(k)}m$. Finally, we note that $\cG_{Q}$ contains $\cG_{P}$, so the same upper bound holds for $R(\cG_P)$.
\end{prf}

In fact, the last result can be made more general. Define a $k \times jk$ matrix to be a $j$-tuple permutation matrix if it is obtained from a permutation matrix by replacing every column with $j$ adjacent copies of itself. In \cite{geneson2009}, it was proved that $\ex(n, R) = O(n)$ for every $j$-tuple permutation matrix $R$. Thus we have the corollary below.

\begin{corollary}
Suppose that $P$ is the maximal $j$-tuple permutation matrix contained in the matrix obtained from applying Lemma \ref{diagatt} $m$ times to any $k \times jk$ $j$-tuple permutation matrix with ones in opposite corners. Then the Ramsey number of $\cG_{P}$ is $O(m)$, where the coefficient of $m$ depends on $k$ and $j$. 
\end{corollary}

In \cite{Conlon2017} it was shown that $R(\cM)\le n^2$ for every interval 2-chromatic ordered matching, and we know that the set of $n$ nested edges (or a sum-decomposable permutation) has Ramsey number of the order $\Theta(n)$. It is natural to ask whether a matching with Ramsey number $\Theta(n^q)$ exists for each exponent $q\in (1,2)$. Indeed, one may adapt the construction used in \cite{Conlon2017} to construct a matching $\cM$ with $R(\cM)=n^{2+o(1)}$ to prove the following.

\begin{theorem}\label{thm:matching-construction}
For each $q\in (1,2)$ there exists an ordered matching $\cM$ on $n$ vertices whose Ramsey number is $n^{q+o(1)}$.
\end{theorem}
\begin{prf}
Let $\pi$ denote the van der Corput permutation on $[n]$, for which it holds that for every pair of intervals $I,J\subseteq [n]$
\begin{align}\label{eq:discrepancy}\left||\pi(I)\cap J|-\frac{|I||J|}{n}\right|\le C\log n\end{align}
for some absolute constant $C$ (see \cite{matousek2009geometric}).

Let $\cM_\pi$ be the ordered matching on $2n^{q-1}$ vertices corresponding to the van der Corput permutation on $[n^{q-1}]$, and let $\cM$ denote the matching obtained by blowing up each edge of $\cM_\pi$ to a matching consisting of $n^{2-q}$ nested edges.\footnote{The term \emph{$m$ nested edges} refers to the ordered graph on $[2m]$ with edge set $\{\{j,n-j+1\}\}_{j=1}^m$} Then $\cM$ is a balanced interval 2-chromatic matching on $2n$ vertices.

We first show that $R(\cM)\le 4n^{q}$.\\
First, we claim that any coloring of the complete graph on $[2n^{2q-2}]$ contains a monochromatic copy of $\cM_\pi$. Indeed, let us split the vertices of $\cK_{2n^{2q-2}}$ into $2n^{q-1}$ intervals of $n^{q-1}$ consecutive vertices. If any two of these intervals induce a monochromatic $K_{n^{q-1},n^{q-1}}$, then this trivially gives a copy of $\cM_\pi$. So between every two of these intervals there is an edge in every color, again giving in fact one copy of $\cM_\pi$ in each color.\\
Now, consider an arbitrary coloring of $\cK_{4n^q}$. Split the vertex set into $2n^{2q-2}$ blocks, each consisting of $2n^{2-q}$ vertices. Between any pair of small intervals, by pigeonhole principle, we may find $n^{2-q}$ nested edges in the same color. Now consider the auxiliary colored complete graph on $2n^{2q-2}$ vertices whose edges are colored according to which color these $n^{2-q}$ nested edges were between the corresponding blocks. By the previous paragraph this contains a monochromatic copy of $\cM_\pi$, which blown back up with the set of nested edges gives a monochromatic copy of $\cM$.

Now, it remains to show that $R(\cM)\ge cn^{q-o(1)}$. Together with the previous paragraph, this implies the result.\\
First, note that a bound similar to \eqref{eq:discrepancy} holds for $\cM$ as well.\\
Let us call each set of $n^{2-q}$ vertices corresponding to a single vertex in $\cM_\pi$ a \emph{block}, and let $I,J\subseteq [n]$ be two intervals. Then $I$ and $J$ each cover at least $\lfloor |I|/n^{2-q}\rfloor-1$ and $\lfloor |J|/n^{2-q}\rfloor-1$ blocks respectively. Then by the properties of $\pi$, we have
\[e(I,n+J)=\left(\frac{(\lfloor|I|/n^{2-q}\rfloor-1)(\lfloor |J|/n^{2-q}\rfloor-1)}{n^{q-1}}\pm C\log n^{q-1}\right)\times n^{2-q}\pm (2n^{2-q}-2)\]
giving (crudely)
\[e(I,n+J)=\frac{|I||J|}{n}\pm \left(2\frac{|I|}{n^{q-1}}+2\frac{|J|}{n^{q-1}}+n^{3-2q}+Cn^{2-q}\log(n^{q-1})+2n^{2-q}+2\right).\]
The trivial bounds $|I|,|J|\le n$ and $6n^{2-q}+n^{3-2q}+2\le (2-q)n^{2-q}\log n$ give 
\[\left|e(I,n+J)-\frac{|I||J|}{n}\right|\le Cn^{2-q}\log n.\]

Now a change of parameters in the construction in \cite{Conlon2017} gives the correct lower bound on the Ramsey number of $\cM$. We include a sketch of the argument with the relevant calculations below.

Let $t=8cn^{q-1}/(\log n)^2 $ and $s=n/8\log n$, where $c\ll C^{-1}$ is a small constant. Consider the complete ordered graph $\cK_t$ with loops, and color at random in red/blue with equal probability. Blow up each vertex to $s$ vertices, blowing up the coloring as well. Our aim is to show that $\PP(\mbox{red copy of }\cM)<1/2$.

Denote the monochromatic intervals by $I_1,...,I_t$, and suppose there is a red copy of $\cM$. Note that only one of these intervals may contain vertices both from the left and right of $\cM$, and all intervals with smaller indices must contain only vertices from the left side of $\cM$ and all intervals with larger indices must contain vertices from the right side of $\cM$. So, we can denote the intersections of the vertex set of $\cM$ with each interval $I_j$ by $A_1,...,A_k,B_k,...,B_t$, where the $A_i$ contain only vertices from the left side of $\cM$ and the $B_i$ contain only vertices from the right side of $\cM$. Note, the sets $A_i, B_i$ are not necessarily intervals. Moreover, the sets $A_i$ and $B_i$ form a partition of $[2n]$ into sets, and the number of such partitions is $\binom{2n+t}{t}$. Note that
\begin{align*}
\binom{2n+t}{t}&\le \left(e\left(\frac{2n}{t}+1\right)\right)^t\le \exp \left\{\frac{8cn^{q-1}}{\log^2 n}\log (c^{-1}n^{2-q}\log^2 n)\right\}
\le e^{8c n^{q-1}/\log n}.
\end{align*}
Fix one such partition.

Let $A$ be the $(d+1)$-st largest set among the $A_i$'s and let $B$ be the $(d+1)$-st largest set among the $B_i$'s, where $d=2\log n$. Then:
\begin{itemize}
    \item If $|A||B|>Cn^{3-q}\log n$, between any $A_i$ larger than $A$ and each $B_i$ larger than $B$ we have strictly more than $Cn^{3-q}\log n/n-Cn^{2-q}\log n=0$ edges of $\cM$. This means that at least $d^2$ edges of the original $K_t$ which we colored randomly have been colored in red. This happens with probability at most $\binom{t}{d}^22^{-d^2}$. We have that
    \begin{align*}
        \binom{t}{d}^22^{-d^2}&\le t^{2d}e^{-d^2} 
        \le e^{4(q-1)\log^2 n-4\log^2 n }<\frac{1}{4}.
    \end{align*}
    \item If $|A||B|\le Cn^{3-q}\log n$, we use the fact that each $A_i$ larger than $A$ contains an endpoint of at most $s$ edges of $\cM$, and each $B_i$ larger than $B$ contains an endpoint of at most $s$ edges of $\cM$. This gives a total of at most $2ds$ edges of $\cM$ coming out of some $A_i$ or $B_i$ larger than $A$ or $B$, leaving at least $n-2ds=n/2$ edges of $\cM$ between small $A$'s and $B$'s. Each such pair $A_i$, $B_j$ has at most $|A_i||B_j|/n+Cn^{2-q}\log n\le 2Cn^{2-q}\log n$ edges of $\cM$ between them, so at least $n^{q-1}/(4C\log n)$ pairs of small sets $(A_i,B_j)$ have at least one edge of $\cM$ between them. This happens with probability \begin{align*}
    \binom{2n+t}{t}2^{-n^{q-1}/(4C\log n)}&\le e^{8cn^{q-1}/\log n-n^{q-1}/(4C\log n)}<\frac1{4}.
    \end{align*}
    
\end{itemize}
\end{prf}

\subsection{The alternating cycle}

In this section we derive an upper bound of the order $k^2$ for the Ramsey number of the alternating ordering of the cycle on $2k$ vertices, as defined below. Note that in \cite{balko2015} it was shown that the monotone ordering of the cycle on $n$ vertices has Ramsey number $2n^2-6n+6$.

\begin{definition}
The \emph{alternating cycle} $\cC_{2k}^{\text{alt}}$ is a interval 2-chromatic ordering of $C_{2k}$ such that $I_1=\{a_1,...,a_k\}$, $I_2=\{b_k,...,b_1\}$, and $E(C_{2k}^{\text{alt}})=\{a_ib_j:|i-j|= 1\}\cup\{a_1b_1,a_kb_k\}$.
\end{definition}

\begin{figure}[H]
	\centering
    \begin{tikzpicture}[yscale = 0.8, scale=.75]
        \foreach \n in {1,...,8} {
			\coordinate (\n) at (\n,0);
            \node (\n) at (\n,0) {};
		}
		\arc{1}{8}
        \arc{1}{7}
        \arc{2}{8}
        \arc{2}{6}
        \arc{3}{7}
        \arc{3}{5}
        \arc{4}{6}
        \arc{4}{5}
;/.		\end{tikzpicture}
\label{fig:altcycle}
\caption{The alternating cycle $\cC_8^{alt}$.}
\end{figure}
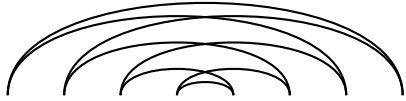

\begin{theorem}
$R(\cC_{2k}^{\text{alt}})= O(k^2)$.
\end{theorem}

\begin{prf}
For ease of notation, let $\cC_{2k}$ denote the alternating cycle $\cC_{2k}^{\text{alt}}$, and let $P_{2k}$ be its associated matrix. We will show that for all $k\ge 2$ we have $\ex(n,P_{2k})\le (2k-2) \ex(n,P_4)$. Then by Lemma~\ref{lem:01mat:ramsey} we have $\left(R(\cC_{2k})-1\right)^2\le 8(2k-2) \ex((R(C_{2k})-1)/2,P_4)$. Note that $P_4=\left(\begin{smallmatrix}1&1\\1&1\end{smallmatrix}\right)$, so by a well-known result on Zarankiewicz's problem due to K\'{o}vari, S\'{o}s and Tur\'{a}n \cite{Kovari1954}, $\ex(n,P_4)=O(n^{3/2})$, giving the result $R(\cC_{2k}^{\text{alt}})=O(k^2)$.\\
It remains to show that $\ex(n,P_{2k})\le (2k-2)\cdot \ex(n,P_4)$. We do this by induction on $k$. The case $k=2$ is trivial. When $k>2$, note that by performing a matrix operation on $P_{2k}$ and $P_{2l}$, we obtain a matrix containing $P_{2(k+l)-2}$. In particular, $\ex(n,P_{2(k+l)-2})\le \ex(n,P_{2k})+\ex(n,P_{2l})$. Then if $k$ is odd, we have $\ex(n,P_{2k})\le 2\ex(n,P_{k+1})\le 2(k-1) \ex(n,P_4)$, and similarly if $k$ is even, we have $\ex(n,P_{2k})\le \ex(n,P_k)+\ex(n,P_{k+2})\le (k-2)\ex(n,P_4)+k\cdot \ex(n,P_4)=(2k-2)\ex(n,P_4)$.
\end{prf}

\begin{remark}
Using the method of flag algebras as described in \cite{lidicky2017+}, we can obtain $R(\cC_{6}^{\text{alt}}) \le 30$.
\end{remark}

\subsection{Extension to hypergraphs}

In the last section, we provided several operations that can be performed on 0-1 matrices to produce interval $2$-chromatic ordered graphs with linear Ramsey numbers. In this section, we generalize those operations to $d$-dimensional 0-1 matrices to produce a class of ordered hypergraphs with linear Ramsey numbers.

Let us define $\ex(n, P, d)$ to be the maximum number of ones in a $P$-free $d$-dimensional 0-1 matrix of sidelength $n$. The proofs for the following operations are analogous to the $2$-dimensional cases proved in \cite{Keszegh2009}, \cite{tardos2005}, and this paper.

\begin{lemma} \label{extend} Let $d\ge 2$ be an integer and let $P$ be a fixed $d$-dimensional matrix.
\begin{enumerate}[(a)]
\item  \label{extend-addlast} If $P'$ is obtained from $P$ by adding a new first $(d-1)$-dimensional hyperplane of entries to $P$ with a single one next to a one of $P$ (see Figure \ref{fig:extend-addlast}), then $\ex(n, P, d) \le \ex(n, P', d) \le \ex(n, P, d)+n^{d-1}$.
\item  \label{extend-addmid} If $P'$ is obtained from $P$ by adding $t$ extra $(d-1)$-dimensional hyperplanes of entries to $P$ between two adjacent $(d-1)$-dimensional hyperplanes of entries of $P$, such that each new $(d-1)$-dimensional hyperplane has a single one in the same $1$-row\footnote{An $i$-row is a maximal set of entries with all coordinates the same except for the $i^{th}$ coordinate.}, all of the new ones have $d-1$ same coordinates, and the newly introduced ones have a one from $P$ next to them on both sides (see Figure \ref{fig:extend-addmid}), then $\ex(n, P, d) \le \ex(n, P', d) \le (t+1)\ex(n, P, d)$.
\item  \label{extend-diagatt} Suppose that $P$ is a $d$-dimensional 0-1 matrix with a one in a corner, and $Q$ is a $d$-dimensional 0-1 matrix with a one in the opposite corner. Let $R$ be the pattern obtained from $P$ and $Q$ by placing them so that the corner ones are in the same location, and the entries outside of $P$ and $Q$ are all zeroes (see Figure \ref{fig:extend-diagatt}). Then $max(\ex(n, P, d), \ex(n, Q, d)) \le \ex(n, R, d) \le \ex(n, P, d)+\ex(n, Q, d)$.
\item  \label{extend-addblanks} Suppose that $P'$ is obtained from $P$ by inserting $k-1$ empty $(d-1)$-dimensional hyperplanes of entries between every adjacent pair of $(d-1)$-dimensional hyperplanes of entries in $P$, as well as before and after the first and last $(d-1)$-dimensional hyperplanes (see Figure \ref{fig:extend-addblanks}). Then $\ex(n, P, d) \le \ex(n, P', d) \le k \ex(n, P, d) + 3d k n$.
\end{enumerate}
\end{lemma}

\begin{figure}[H]
\centering
\caption{}
\label{fig:extend}
    \begin{subfigure}[H]{0.3\textwidth}
        \centering
        \includegraphics{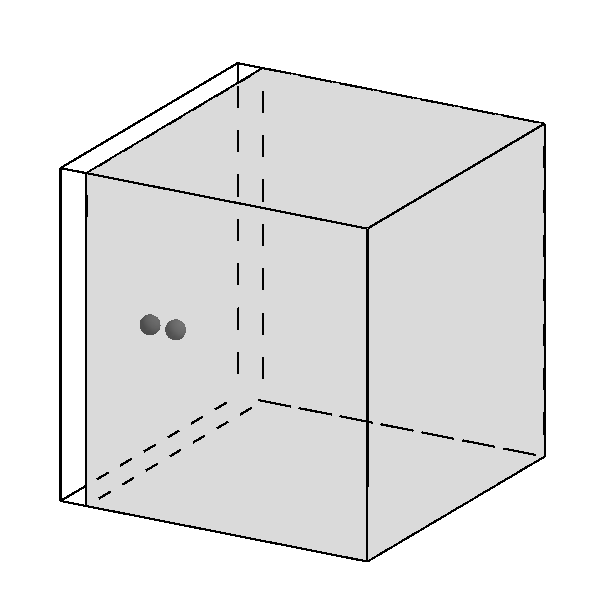}
        \caption{Operation for Lemma \ref{extend}\ref{extend-addlast}}
        \label{fig:extend-addlast}
    \end{subfigure}
    ~
    \begin{subfigure}[H]{0.3\textwidth}
        \centering
        \includegraphics{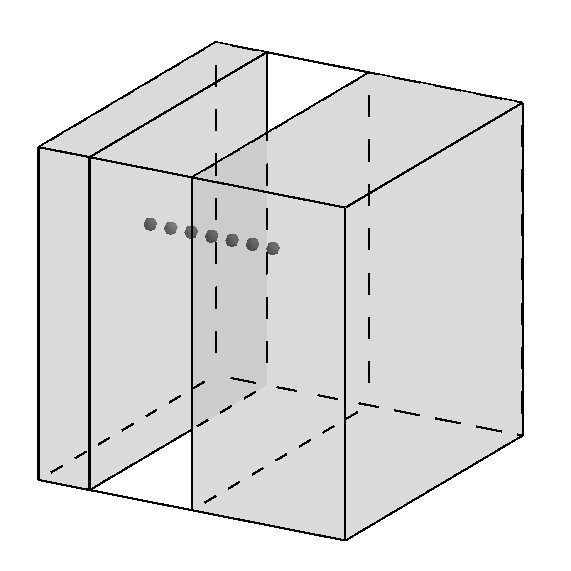}
        \caption{Operation for Lemma \ref{extend}\ref{extend-addmid}}
        \label{fig:extend-addmid}
    \end{subfigure} 
    ~
    \begin{subfigure}[H]{0.3\textwidth}
        \centering
        \includegraphics{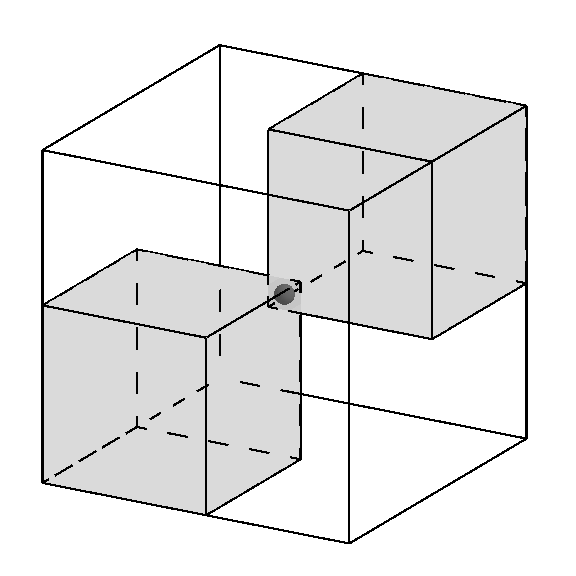}
        \caption{Operation for Lemma \ref{extend}\ref{extend-diagatt}}
        \label{fig:extend-diagatt}
    \end{subfigure}    
    \\
    \vspace{10pt}
    \begin{subfigure}[H]{0.3\textwidth}
        \centering
        \includegraphics{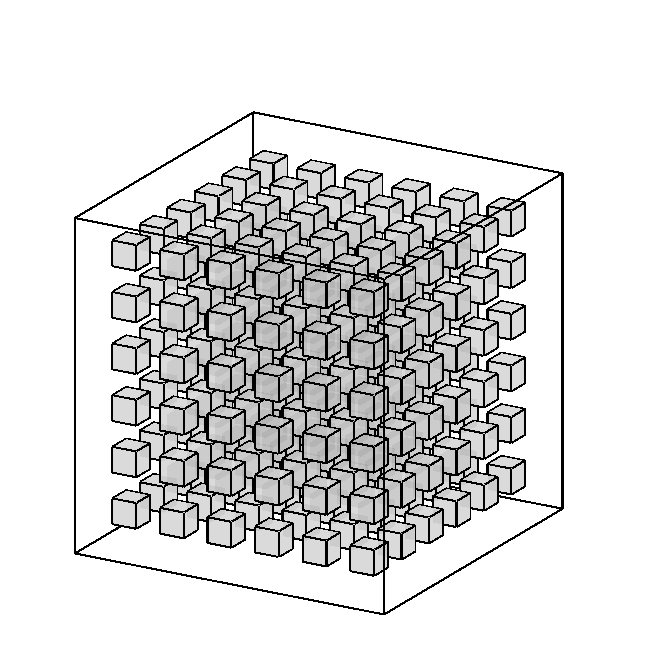}
        \caption{Operation for Lemma \ref{extend}\ref{extend-addblanks}}
        \label{fig:extend-addblanks}
    \end{subfigure}    
    ~
    \begin{subfigure}[H]{0.3\textwidth}
        \centering
        \includegraphics{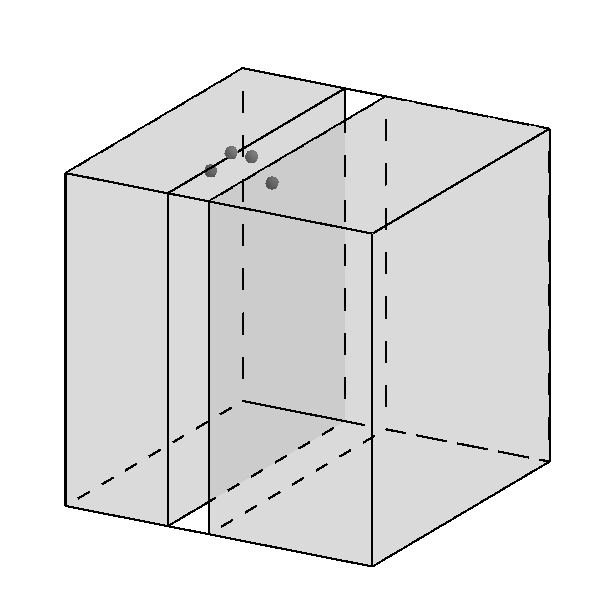}
        \caption{Operation for Lemma \ref{extendpettie}}
        \label{fig:extend-pettie}
    \end{subfigure}  
\end{figure}

The extension of the operation detailed in Lemma \ref{addmidup} (see \cite{Keszegh2009} for proof) to $d$-dimensional 0-1 matrices is less obvious to prove from the $2$-dimensional case than the operations in the preceding lemma, but the proof for $d$ dimensions can be done in a way that is very similar to the proof of the grafting lemma in \cite{Pettie2011}. 

For the proof below, we define $\ex(n, m, P)$ to be the maximum number of ones in a $P$-free 0-1 matrix with dimensions $n \times m$.

\begin{lemma} \label{extendpettie}
Assume that $P$ has two adjacent ones in a $1$-row with minimal $d^{th}$ coordinate (we call this a top $1$-row). Let $Q$ be obtained from $P$ by first adding two empty $(d-1)$-dimensional hyperplanes $A$ and $B$ of entries between the $(d-1)$-dimensional hyperplanes with the ones in the top $1$-row, and then adding a new $(d-1)$-dimensional hyperplane of entries directly above $P$ with exactly two ones that are in the $(d-1)$-dimensional hyperplanes $A$ and $B$ and directly above the top $1$-row from $P$ with the two adjacent ones (see Figure \ref{fig:extend-pettie}). Then $\ex(n, Q, d) = O(\ex(n, P, d))$.
\end{lemma}

\begin{prf}
Let $A$ be a $Q$-free $d$-dimensional 0-1 matrix of sidelength $n$ with weight $\ex(n, Q, d)$. First partition the matrix $A$ into $n^{d-2}$ $2$-dimensional hyperplanes $A_{1}, \dots, A_{n^{d-2}}$ of entries that only have length greater than $1$ in the $1^{st}$ and $d^{th}$ dimensions. 

For each matrix $A_{i}$, partition the ones in each $1$-row of $A_{i}$ into consecutive blocks of $g$ ones for some $g > 7$, leaving up to $(g-1)n$ ones ungrouped, $g-1$ per $1$-row. Form a matrix $A'_{i}$ from $A_{i}$ by assigning each block to a distinct $1$-row in this way: Let $b_{i,j}$ be the number of blocks in $1$-row $j$ of $A_{i}$ and let $b_{i, <j} = \sum_{j' < j} b_{i, j'}$ be the number of blocks in $1$-rows preceding $j$. Block $t$ of $1$-row $j$ of $A_{i}$ is assigned to row $b_{i, <j}+t$ of $A'_{i}$. 

As in \cite{Pettie2011}, call a one in $A'_{i}$ good if it is in the bottom left corner of a copy of the pattern $R$ below: 

$$
\begin{pmatrix}
        & \bullet & \bullet \\
\bullet &         &         & \bullet\\
\end{pmatrix}
$$

A $1$-row in $A'_{i}$ is called bad if it has no good ones. The submatrix of bad $1$-rows is $R$-free and if there are $b$ bad $1$-rows, then there are exactly $b g$ ones in this submatrix, so $b g \le \ex(b, n, R) \le 7b+7n$. The last inequality is well-known, see e.g. \cite{fulek2009}. Thus we have $b \le 7n/(g-7)$. 

Also form a new $n \times \dots \times n \times m$ $d$-dimensional 0-1 matrix $A_{\text{good}}$ from $A$ that contains exactly a single one from each block in $A$ that corresponds to a good one in some $A'_{i}$. Then $A_{\text{good}}$ avoids $P$, or else $A$ contained $Q$.

Thus we have $w(A) \le g w(A_{\text{good}})+(g-1)n^{d-1}+7 g n^{d-1} / (g-7) \le g \ex(n, P, d) + (g-1)n^{d-1}+7 g n^{d-1} / (g-7) = O(\ex(n, P, d))$ by setting e.g. $g = 8$.
\end{prf}

We now generalize from ordered graphs to ordered hypergraphs.
A $d$-dimensional 0-1 matrix is equivalent to a $d$-partite, $d$-uniform ordered hypergraph, where the number of vertices in the $i^{th}$ partition is equal to the size of the $i^{th}$ dimension of the matrix. Each one in the matrix at index $(i_1, \dots, i_d)$ corresponds to an edge in the hypergraph on vertices $i_1, \dots, i_d$. Klazar and Marcus \cite{klazar2007} derived several results on ordered hypergraphs using bounds on permutation matrices.

The following is a generalization of Theorem 18 to the hypergraph setting.

\begin{theorem} \label{thmoperd}
Let $P$ be a fixed $d$-dimensional 0-1 matrix with $\ex(n, P, d) = O(n^{d-1})$. Suppose that we perform one of the operations below to yield a new family of $d$-dimensional 0-1 matrices $\left\{P_{j}\right\}$. Then the Ramsey numbers of the ordered hypergraphs $\cG_{P_{j}}$ grow linearly with respect to their size.

\begin{enumerate}
\item Assume that $P$ has a one in its first or last $(d-1)$-dimensional hyperplane of entries in some dimension. Let $P_j$ be obtained from $P$ by iterating the operation in Lemma \ref{extend}\ref{extend-addlast} a total of $j$ times.

\item Assume that $P$ has two adjacent ones in some $1$-row. Let $P_j$ be obtained from $P$ by applying the operation in Lemma \ref{extend}\ref{extend-addmid} with $t = j$.

\item Assume that $P$ has two adjacent ones in a top $1$-row. Let $P'$ be obtained from $P$ by applying the operation in Lemma \ref{extendpettie} to obtain a new $d$-dimensional 0-1 matrix with two adjacent ones in a $1$-row in a new top $(d-1)$-dimensional hyperplane of entries, and then let $P_j$ be obtained from $P'$ by applying the operation in Lemma \ref{extend}\ref{extend-addmid} with $t = j$ to those adjacent ones in the $1$-row in the new top $(d-1)$-dimensional hyperplane of entries.

\item \label{diagd} Assume that $P$ has ones in opposite corners. Let $P_j$ be obtained from $P$ by applying the operation in Lemma \ref{extend}\ref{extend-diagatt} a total of $j$ times with $P = Q$.

\item For this last operation, there are no additional requirements on $P$. Let $P_j$ be obtained from $P$ by inserting $n$ empty $(d-1)$-dimensional hyerplanes between every adjacent pair of $(d-1)$-dimensional hyperplanes in $P$.
\end{enumerate}
\end{theorem}

We mention some corollaries of the preceding theorem that are analogous to some of the corollaries from the last section. Suppose that $P$ is any $d$-dimensional permutation matrix of sidelength $k$ with ones in opposite corners. In \cite{klazar2007}, it was proved that $\ex(n, P, d) = O(n^{d-1})$, and later in \cite{geneson2017} it was showed that $\ex(n, P, d) = 2^{O(k)}n^{d-1}$, where the coefficient in the $O(k)$ depends on $d$. 

A straightforward extension of Lemma \ref{lem:01mat:ramsey} gives us a natural $d$-dimensional generalisation of Proposition \ref{prop:layeredperm}:

\begin{corollary}
Let $P$ be a sum-decomposable $d$-dimensional permutation matrix of sidelength $km$, where each block has sidelength $k$. Then the Ramsey number of $\cG_P$ is $2^{O(k)}m$ where the coefficient in $O(k)$ depends on $d$.
\end{corollary}

We can also define a $d$-dimensional $j$-tuple permutation matrix as the Kronecker product of a $d$-dimensional permutation matrix and a $d$-dimensional matrix of all ones where only one dimension has length greater than $1$. In \cite{geneson2017}, it was proved that $\ex(n, Q, d) = 2^{O(k)}n^{d-1}$ for every $d$-dimensional $j$-tuple permutation matrix in which $d-1$ of the dimensions have sidelength $k$, where the coefficient in the $O(k)$ depends on $d$ and $j$.

Suppose that $Q$ is a $d$-dimensional $j$-tuple permutation matrix with $j k$ ones that has ones in opposite corners. If $Q'$ is the maximal $d$-dimensional $j$-tuple permutation matrix contained in the $d$-dimensional 0-1 matrix obtained from applying Lemma \ref{extend}\ref{extend-diagatt} $m$ times, then the ordered hypergraph $\cG_{Q'}$ has Ramsey number $O(m)$. 

One specific ordered hypergraph whose Ramsey number has been of interest in recent years is the \emph{monotone hyperpath} $\cP^d_n$ on $n$ vertices with edge-set consisting of all intervals of $d$ consecutive vertices. In \cite{moshkovitz2014} it was shown that the $t$-color Ramsey number of $\cP^d_n$ for $d\ge 3$ is a tower of height $d-2$ whose final exponent is between $(n-d+1)^{t-1}/2\sqrt{t}$ and $2(n-d+1)^{t-1}$.  We know that in the simple graph case, the monotone path has relatively high Ramsey number (quadratic in the number of vertices), while the "special" alternating path has linear Ramsey number. We show that similar behaviour can be observed in the hypergraph setting as well, where the Ramsey number is high for the monotone hyperpath $\cP^d_n$, but quite low for a natural analogue of an alternating path.

\begin{definition}
The tight $d$-uniform hyperpath on $n$ vertices $P_n^{d}$ has vertex set $[n]$ and edges of the form $\{j,j+1,...,j+d-1\}$ for $1\le j\le n-d+1$.\\
The alternating $d$-partite ordering $\cA_n^{d}$ of the tight $d$-uniform hyperpath on $n=dm$ vertices is $1,d+1,2d+1,...,(m-1)d+1,2,d+2,...,(m-1)d+2,...,d,2d,...,md$.
\end{definition}

\begin{minipage}{0.4\textwidth}
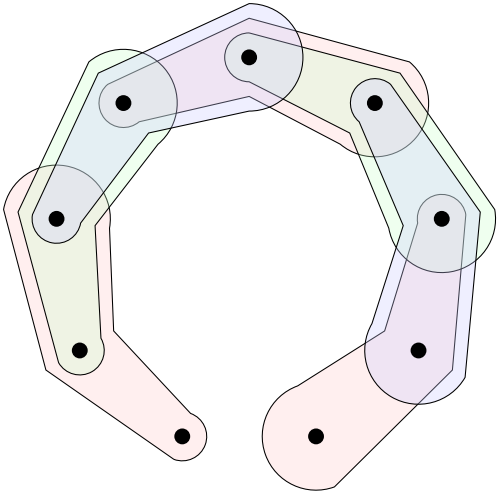
\begin{figure}[H]
	\centering
    \begin{tikzpicture}[scale = 2.6]
		\foreach \n in {1,4,7} {
            \draw[fill=red!20,fill opacity=0.3] ({290-40*\n}:1.125) -- ({250-40*\n}:1.2) -- ({210-40*\n}:1.275) arc ({210-40*\n}:{30-40*\n}:0.275) -- ({250-40*\n}:0.8) -- ({290-40*\n}:0.875) arc ({110-40*\n}:{-70-40*\n}:0.125);
		}
		\foreach \n in {2,5} {
            \draw[fill=green!20,fill opacity=0.3] ({290-40*\n}:1.125) -- ({250-40*\n}:1.2) -- ({210-40*\n}:1.275) arc ({210-40*\n}:{30-40*\n}:0.275) -- ({250-40*\n}:0.8) -- ({290-40*\n}:0.875) arc ({110-40*\n}:{-70-40*\n}:0.125);
		}
		\foreach \n in {3,6} {
            \draw[fill=blue!20,fill opacity=0.3] ({290-40*\n}:1.125) -- ({250-40*\n}:1.2) -- ({210-40*\n}:1.275) arc ({210-40*\n}:{30-40*\n}:0.275) -- ({250-40*\n}:0.8) -- ({290-40*\n}:0.875) arc ({110-40*\n}:{-70-40*\n}:0.125);
		}
        
        \foreach \n in {1,...,9} {
            \node[circle,fill=black,inner sep=0pt,minimum size=6pt] (n\n) at ({290-40*\n}:1) {};
		}
\end{tikzpicture}
\label{fig:althypercycle1}
\caption{$P_9^{3}$, the tight $3$-uniform hyperpath on 9 vertices. Hyperedges are drawn as blobs.}
\end{figure}
\end{minipage}
\quad
\begin{minipage}{0.52\textwidth}
\begin{figure}[H]
	\centering
    \begin{tikzpicture}[scale = 0.95]
        \def\xx{1}
        \def\yy{1.2}
        \foreach \n/\m in {1/1,2/3,3/5,4/1,5/3,6/5,7/1,8/3,9/5} {
            \node[circle,fill=black,inner sep=0pt,minimum size=6pt] (n\n) at ({\xx*\n},{\yy*\m}) {};
        }
        \foreach \n in {1,...,3} {
            \draw
            ({3*\xx*\n-\xx*2.5},{\yy*-0.25}) -- 
            ({3*\xx*\n+\xx*0.5},{\yy*-0.25}) -- 
            ({3*\xx*\n+\xx*0.5},{\yy*0.25}) -- 
            ({3*\xx*\n-\xx*2.5},{\yy*0.25}) -- cycle;
        }
        \foreach \n in {1,...,9} {
            \node (m\n) at ({\xx*\n},{0}) {\n};
            \draw[dashed,->] (m\n) -- (n\n);
        }
        \foreach \n in {1,...,3} {
            \draw[fill=red!20,fill opacity=0.3]
            ({\xx*\n},      {-\yy+\yy*2*\n+0.3125}) --
            ({\xx*\n+\xx*3},{-\yy+\yy*2*\n+0.5}) -- 
            ({\xx*\n+\xx*6},{-\yy+\yy*2*\n+0.6875}) arc (90:-90:0.6878) --
            ({\xx*\n+\xx*3},{-\yy+\yy*2*\n-0.5}) -- 
            ({\xx*\n},      {-\yy+\yy*2*\n-0.3125}) arc (270:90:0.3125);
		}
		\foreach \n in {1,...,2} {
		    \draw[fill=green!20,fill opacity=0.3]
            ({\xx*\n+\xx*3},        {-\yy+\yy*2*\n+0.3125}) -- 
            ({\xx*\n+\xx*6-1}, {-\yy+\yy*2*\n}) -- 
            ({\xx*\n+\xx},          { \yy+\yy*2*\n-0.6875}) arc (270:45:0.6875) --
            ({\xx*\n+\xx*6+0.3536}, {-\yy+\yy*2*\n+0.3536}) arc (45:-90:0.5) --
            ({\xx*\n+\xx*3},        {-\yy+\yy*2*\n-0.3125}) arc (270:90:0.3125);
		}
		\foreach \n in {1,...,2} {
		    \draw[fill=blue!20,fill opacity=0.3]
		    ({\xx*\n+\xx*6-0.3125}, {-\yy+\yy*2*\n}) --
		    ({\xx*\n+\xx},          { \yy+\yy*2*\n-0.5}) arc (270:90:0.5) --
		    ({\xx*\n+\xx*4},        { \yy+\yy*2*\n+0.6875}) arc (90:-90:0.6875) --
		    ({\xx*\n+\xx+1},   { \yy+\yy*2*\n-0.3536}) -- 
		    ({\xx*\n+\xx*6},        {-\yy+\yy*2*\n+0.3125}) arc (90:-180:0.3125);
		}
\end{tikzpicture}
\label{fig:althypercycle2}
\caption{$\cA_9^{3}$, the alternating tripartite ordering of $P_9^{3}$. Vertices have been moved vertically so the hyperedges are easier to see.}
\end{figure}
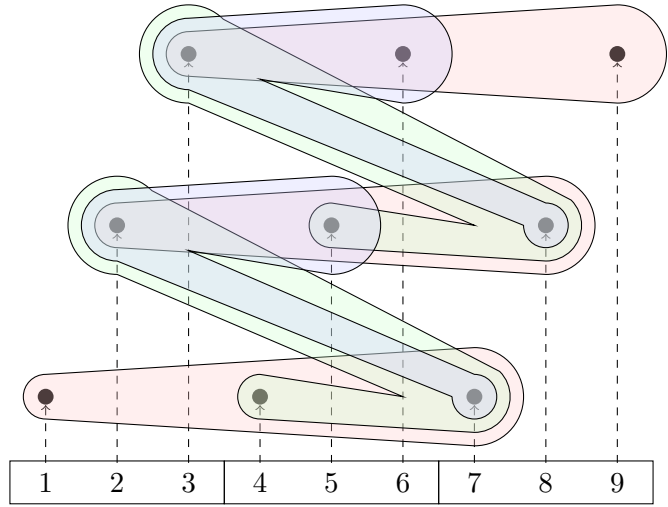
\end{minipage}

\begin{proposition}
For integers $n$ divisible by $d$, it holds that $R(\cA_{n}^{d})\le 2dn$.
\end{proposition}

\begin{prf}
The $d$-dimensional matrix $M$ corresponding to the $d$-partite canonical ordering consists of zeros and a path of ones starting at $\mathbf{x_0}=(1,...,1)\in [n/d]^d$ and satisfying $\mathbf{x_j}=\mathbf{x_{j-1}}+\delta_{j'}$, where $\delta_{j'}\in [n/d]^d$ consists of zeros and a one at position $j'\in\{1,...,d\}$ such that $j'\equiv j \mod d$. Note that $\mathbf{x_{n-1}}=(n/d,....,n/d)\in [n/d]^d$ and the matrix can be obtained by $n-1$ successive applications of the operation described in Lemma \ref{extend}\ref{extend-addlast} to the identity matrix of side length 1. Therefore $\ex(N,M,d)\le nN^{d-1}$. By a straightforward extension of Lemma \ref{lem:01mat:ramsey} we get that $\frac{1}{2}\left(R(\cA_{n}^{d})/d\right)^d\le \ex(R(\cA_{n}^{d})/d,\cA_{n}^{d},d)$, which implies the result.
\end{prf}
\vspace{-5pt}
\section{Concluding remarks}

The problem of determining which orderings of a given graph have high or low Ramsey numbers is an interesting one, and a lot is still unknown. 

For example, the path $P_n$ on $n$ vertices has a monotone ordering whose Ramsey number (see Proposition \ref{addlastedge} or \cite{fox2012}) is $(n-1)^2+1$, and an alternating ordering (see \cite{balko2015}) whose Ramsey number is $O(n)$. In light of this, we mention a conjecture of Balko et al \cite{balko2015}:
\begin{conjecture}
Among all orderings of $P_n$, the alternating path has minimum Ramsey number.
\end{conjecture}

The bound from Theorem \ref{thm:randommatching} implies that there exists an ordering of the path whose Ramsey number is $n^{\Omega(\log n/\log \log n)}$. In the unordered setting, the Ramsey numbers of matchings and paths are of the same order, which prompts us to ask the following question:

\begin{question}
Is it true that there exists an ordering $\cP_n$ of the path on $n$ vertices, such that $R(\cP_n)=n^{\omega(\log n/\log \log n)}$?
\end{question}

On a similar note, it is known that every interval 2-chromatic ordering $\cM$ of the matching on $n$ vertices has Ramsey number at most $n^2$ (with a matching lower bound \cite{Conlon2017}), so 

\begin{question}
Is it true that every interval 2-chromatic ordering $\cP_n$ of the path on $n$ vertices has Ramsey number $O(n^2)$?
\end{question}
\vspace{-10pt}
\section{Acknowledgements}
We would like to thank Bernard Lidick\'y for discussions during the early stages of this project, and for his comments on this manuscript. This work was initiated at the Graduate Research Workshop in Combinatorics supported by NSF grants 1604458, 1604697, 1603823, "Collaborative Research:  Rocky Mountain - Great Plains Graduate Research Workshops in Combinatorics" and NSA grant H98230-18-1-0017, "The 2018 and 2019 Rocky Mountain - Great Plains Graduate Research Workshops in Combinatorics".  

\addcontentsline{toc}{section}{References}

\bibliographystyle{abbrv}
\bibliography{references}

\end{document}